\documentclass[12pt, reqno]{amsart}
\usepackage{amssymb}
\usepackage[mathcal]{euscript}
\usepackage[cmtip,all]{xy}
\usepackage{amsmath}\usepackage{amsfonts}
\usepackage{graphicx} 
\usepackage{verbatim} 
\usepackage{tikz}
\usepackage{hyperref} 
\usepackage[left=20mm, right=20mm, top=20mm, bottom=20mm]{geometry} 

\newcommand{\bi}{\textnormal{\textbf{i}}}  \newcommand{\cB}{\mathcal{B}}  \newcommand{\FF}{\mathbb{F}} \newcommand{\ZZ}{\mathbb{Z}} \newcommand{\NN}{\mathbb{N}} 
\newcommand{\cT}{\mathcal{T}}  \newcommand{\cV}{\mathcal{V}} \newcommand{\cW}{\mathcal{W}} \newcommand{\cH}{\mathcal{H}} \newcommand{\cS}{\mathcal{S}} \newcommand{\med}{\;|\;} 
 \DeclareMathOperator{\Aut}{\mathrm{Aut}}     
 
 \DeclareMathOperator{\bAut}{\mathbf{Aut}}
\DeclareMathOperator{\bPGL}{\mathbf{PGL}} 
\DeclareMathOperator{\bmu}{\boldsymbol{\mu}} 
\DeclareMathOperator{\bGL}{\mathbf{GL}} \DeclareMathOperator{\bSL}{\mathbf{SL}}
\DeclareMathOperator{\bGO}{\mathbf{GO}} \DeclareMathOperator{\GO}{\mathrm{GO}}
\DeclareMathOperator{\chr}{\mathrm{char}\,} \DeclareMathOperator{\TKK}{\mathrm{TKK}}
 \DeclareMathOperator{\gl}{\mathfrak{gl}}
\DeclareMathOperator{\psl}{\mathfrak{psl}}

\newcommand{\Ort}{\mathrm{O}} \newcommand{\bOrt}{\mathbf{O}} 

\newcommand{\GL}{\mathrm{GL}} 
  \newcommand{\cM}{\mathcal{M}}
\newcommand{\cA}{\mathcal{A}} 
 \DeclareMathOperator{\id}{\mathrm{id}}

\DeclareMathOperator{\im}{\textnormal{im}}

\newcommand{\Lie}{\mathrm{Lie}}

\newcommand{\bG}{\mathbf{G}}
\newcommand{\bH}{\mathbf{H}}
\newcommand{\cJ}{\mathcal{J}}
\newcommand{\Tr}{\mathsf{T}} 
\DeclareMathOperator{\tr}{\textnormal{t}} 
\DeclareMathOperator{\diag}{\mathrm{diag}} 
 
\DeclareMathOperator{\bT}{\mathbf{T}}

\DeclareMathOperator{\Ad}{\mathrm{Ad}}
\newcommand{\cR}{\mathcal{R}}
\DeclareMathOperator{\bStab}{\mathbf{Stab}}
\newcommand{\Sim}{\mathrm{Sim}} \newcommand{\bSim}{\mathbf{Sim}} 
\newcommand{\Iso}{\mathrm{Iso}} \newcommand{\bIso}{\mathbf{Iso}} 
\newcommand{\VtI}{\widetilde{\cV}^{\textnormal{(I)}}}
\newcommand{\VhI}{\cV^{\textnormal{(I)}}}
\newcommand{\TtI}{\widetilde{\cT}^{\textnormal{(I)}}} 
\newcommand{\ThI}{\cT^{\textnormal{(I)}}}
\newcommand{\VIV}{\cV^{\textnormal{(IV)}}}
\newcommand{\TIV}{\cT^{\textnormal{(IV)}}}
\newcommand{\ThIV}{\widehat{\cT}^{\textnormal{(IV)}}}

\DeclareMathOperator{\AlgF}{\mathrm{Alg_{\FF}}} 
\newcommand{\cL}{\mathcal{L}}
\renewcommand{\det}{\mathrm{det}} 

\newtheorem{theorem}{Theorem}
\newtheorem{proposition}[theorem]{Proposition}

\newtheorem{corollary}[theorem]{Corollary}

\theoremstyle{definition}

\newtheorem{notation}[theorem]{Notation}
\newtheorem{remark}[theorem]{Remark}
\newtheorem{remarks}[theorem]{Remarks}

\numberwithin{equation}{section} 
\numberwithin{theorem}{section} 

\begin{document}

\title[Automorphism group schemes of special simple Jordan pairs of types I and IV]
{Automorphism group schemes of \\ special simple Jordan pairs of types I and IV}

\author[D. Aranda-Orna]{Diego Aranda-Orna}
\address{Departamento de Matem\'{a}ticas,
Universidad de Oviedo, 33007 Oviedo, Spain}
\email{diego.aranda.orna@gmail.com}

\author[A. Daza-Garc\'ia]{Alberto Daza-Garc\'ia}
\address{Departamento de Matem\'{a}ticas Aplicadas I,
	Universidad de Sevilla, 41012 Sevilla, Spain}
\email{adaza1@us.es}

\thanks{Both authors are supported by grant PID2021-123461NB-C21, funded by MCIN/AEI/10.13039/501100011033 and by ``ERDF A way of making Europe''.}
\date{}

\begin{abstract}
In this work, the automorphism group schemes of finite-dimensional simple Jordan pairs of types I and IV, and of some Jordan triple systems related to them, are determined. We assume $\chr\FF \neq 2$ for the base field $\FF$.
\end{abstract}

\maketitle

\section{Introduction}
We will always assume that the base field $\FF$ has characteristic different from $2$.

\smallskip

We assume that the reader is familiar with affine group schemes, for which one may consult \cite[Appendix A]{EKmon} or \cite{W79}. Recall from \cite[Theorems 1.38 \& 1.39]{EKmon} that, for two algebras with isomorphic automorphism group schemes, we can transfer the classifications of gradings up to isomorphism, and the classification of fine gradings up to equivalence. Due to the nature of the proof, this also holds if we include other structures such as Jordan triple systems and Jordan pairs (see as an example \cite{D24}). This is the main motivation for the work of the present paper.

\bigskip

The classes of finite-dimensional simple Jordan pairs over an algebraically closed field $\FF$ (see \cite{L75}) are:
\begin{itemize}
	\item ($\text{I}_{m,n}$): $(\cM_{m\times n}(\FF),\cM_{m\times n}(\FF))$, $m\times n$ matrices over $\FF$.
	\item ($\text{II}_n$): $(\cA_n(\FF),\cA_n(\FF))$, $n\times n$ alternating matrices over $\FF$.
	\item ($\text{III}_n$): $(\cH_n(\FF),\cH_n(\FF))$, $n\times n$ hermitian matrices over $\FF$.
	\\
	For the three types above, the quadratic products are given by $Q^\sigma_x(y) := x y^\Tr x$.	
	\item($\text{IV}_n$): $(\FF^n,\FF^n)$ with quadratic products given by $Q^\sigma_x(y)=b(x,y)x-q(x)$, where $b$ is the standard scalar product and $q(x):=\frac{1}{2}b(x,x)$.
	\item($\text{V}$): the bi-Cayley pair.
	\item ($\text{VI}$): the Albert pair.	
\end{itemize}
In the list above, the bi-Cayley and Albert pairs are \emph{exceptional}, and the others (i.e., the pairs of types $\text{I}_{m,n}$, $\text{II}_{n}$, $\text{III}_{n}$ and $\text{IV}_{n}$) are called \emph{special}. The elements of the automorphism groups of finite-dimensional special simple Jordan pairs were described in \cite{S84}, although the structure of these automorphism groups was not given explicitly.

Recall that given a Jordan pair $(\cV^+,\cV^-)$ with trilinear products denoted by $\{\cdot,\cdot,\cdot\}^{\sigma}$ ($\sigma = \pm$) and an involution $(\iota^+, \iota^-)$, it is possible to define a Jordan triple system $(\cV^+,\{\cdot,\cdot,\cdot\})$ where the triple product is given by
$$ \{x,y,z\} := \{x,\iota^+(y),z\}^+ $$
for all $x,y,z\in \cV^+$. And vice versa, for a Jordan triple system $\cT$ with trilinear product $\{\cdot,\cdot,\cdot\}$ we can define a Jordan pair $(\cT,\cT)$ with involution $(\id_\cT, \id_\cT)$, where the trilinear products are given by
$$ \{x,y,z\}^{\sigma} := \{x,y,z\}. $$
A classification (in terms of involutions) of the finite-dimensional simple special Jordan triple systems over an arbitrary field is given in \cite{S85}; unfortunately, since that classification is more general, it does not provide an explicit list of isomorphy classes when the field is algebraically closed.

\bigskip

Our first purpose is to determine the automorphism group schemes of the finite-dimensional simple Jordan pairs of types I and IV, as well as to find the automorphism group schemes of some related Jordan triple systems. In both cases (types I and IV), one of the Jordan triple systems considered in this work (the triple systems $\TtI_{m,n}$ in Notation~\ref{notationPairsTypeI} and $\ThIV(W, b)$ in Notation~\ref{notationType4}) is the one associated to the identity involution of the Jordan pair of the same type, and the other one (the triple systems $\ThI_n$ in Notation~\ref{notationPairsTypeI} and $\TIV(V, b)$ in Notation~\ref{notationType4}) is a Jordan triple system constructed from a Jordan algebra (the two simplest cases).

For the simple Jordan algebras considered in this work, the automorphism group schemes are already well-known (but the results related to automorphism group schemes of Jordan triple systems and Jordan pairs are all original). For simple Jordan pairs of type I, in the case of square matrices, the automorphism group was studied in \cite[Th.7]{J76}; we extend that result to automorphism group schemes, and we also calculate the automorphism group schemes in the case of (nonsquare) rectangular matrices, which covers all cases of type I.

\medskip

This paper is structured as follows. In \S\ref{section.preliminaries} we establish some basic definitions and notation that will be used in further sections. In \S\ref{section.type.IV} and \S\ref{section.type.I}, respectively, the automorphism group schemes of the simple Jordan pairs of types IV and I (and of some related Jordan triple systems) will be determined. Our main results for Jordan pairs are Propositions \ref{AutV.IV}, \ref{automorphismgroupscheme.I.square} and \ref{automorphismgroupscheme.I.rectangle}, whereas for Jordan triple systems these are Corollaries \ref{AutT.IV}, \ref{Aut.TJI}, \ref{Aut.tildeTImn}, \ref{c:fin1}. As an application we show in Corollary \ref{c:fin} that for $n>1$ the Jordan triple systems $\TtI_n$ and $\ThI_n$ (both of which induce isomorphic Jordan pairs) have non-isomorphic automorphism group schemes, what in particular provides an alternative proof to \cite{S85} of the fact that they are not isomorphic Jordan triple systems.

\section{Preliminaries} \label{section.preliminaries}

In this section we will recall the basic definitions of Jordan systems that will be used through the document. 

\medskip

A {\em (linear) Jordan pair} is a pair of vector spaces $\cV=(\cV^+,\cV^-)$ with two trilinear products $\cV^\sigma \times \cV^{-\sigma} \times \cV^\sigma \to \cV^\sigma $, $(x,y,z) \mapsto \{x,y,z\}^\sigma$ satisfying:
\begin{align}
& \{x,y,z\}^\sigma = \{z,y,x\}^\sigma, \\
& [D^\sigma_{x,y}, D^\sigma_{u,v}]
= D^\sigma_{D^\sigma_{x,y}u, v} - D^\sigma_{u, D^{-\sigma}_{y,x}v},
\end{align}
where the $D$-operators are defined by $D^\sigma_{x,y}(z) := \{x,y,z\}^\sigma$. The superscript $\sigma \in \{+, - \}$ is sometimes omitted to simplify the notation. Denote $Q_x(y) := \frac{1}{2}\{x,y,x\}$, which is quadratic on $x \in \cV^\sigma$ and linear on $y \in \cV^{-\sigma}$.

The classification of finite-dimensional simple Jordan pairs over an algebraically closed field can be found in \cite[Chapter~4]{L75}. There is a variation of the TKK-construction of a Lie algebra, starting from a Jordan pair, which the reader may consult in \cite[\S7]{LN19}.

\medskip

A {\em (linear) Jordan triple system} is a vector space $\cT$ with a trilinear product $\cT\times\cT\times\cT\to\cT$, $(x,y,z) \mapsto \{x,y,z\}$, satisfying:
\begin{align}
& \{x,y,z\} = \{z,y,x\}, \\
& [D_{x,y}, D_{u,v}] = D_{D_{x,y}u, v} - D_{u, D_{y,x}v},
\end{align}
where $D_{x,y}(z) := \{x,y,z\}$. Its associated Jordan pair is given by $\cV_\cT := (\cT, \cT)$ and two copies of the triple product.

\medskip

Recall that a \emph{(linear) Jordan algebra} is a (nonassociative) commutative algebra $\cJ$ satisfying the identity $(x^2 y)x = x^2 (yx)$ for all $x,y\in\cJ$. The Jordan triple system associated to $\cJ$ is defined by $\cT_\cJ := \cJ$ with the triple product $\{x,y,z\} := (xy)z + (zy)x - (zx)y$. (This triple product is sometimes scaled by a factor of $2$, which if $\FF$ is algebraically closed it doesn't affect much because in that case both triple systems are isomorphic.) The Jordan pair $\cV_\cJ$ associated to $\cJ$ is defined as the Jordan pair associated to $\cT_\cJ$.

\medskip

An \emph{automorphism} of a Jordan pair $\cV$ is a pair of maps $\varphi=(\varphi^+,\varphi^-) \in \GL(\cV^+) \times \GL(\cV^-)$ such that $\varphi^\sigma(\{x,y,z\}) = \{\varphi^\sigma(x), \varphi^{-\sigma}(y), \varphi^\sigma(z)\}$ for all $x,z\in\cV^\sigma$, $y\in\cV^{-\sigma}$, $\sigma \in \{+,-\}$. Similarly, given a Jordan triple system $\cT$, a map $\varphi \in \GL(\cT)$ is said to be an \emph{automorphism} of $\cT$ if $\varphi(\{x,y,z\}) = \{\varphi(x), \varphi(y), \varphi(z)\}$ for all $x,y,z\in\cT$, or equivalently, if $(\varphi, \varphi)$ is an automorphism of the associated Jordan pair $\cV_\cT$. Note that we can identify $\Aut(\cT) \leq \Aut(\cV)$ through the group monomorphism $\varphi \mapsto (\varphi, \varphi)$.

\bigskip

Let $\AlgF$ denote the category of commutative associative unital $\FF$-algebras. Take $\cR \in \AlgF$. For a Jordan pair $\cV$, we denote $\cV_\cR := (\cV^+_\cR, \cV^-_\cR)$ where $\cV^\sigma_\cR = \cV^\sigma \otimes \cR$, and triple products extended by $\cR$-linearity. The \emph{automorphism group scheme} of $\cV$ is defined by $\bAut(\cV)(\cR) := \Aut_\cR(\cV_\cR)$. Note that $\bAut(\cV)$ is an affine group scheme (which also applies to morphisms in $\AlgF$). Similarly, for a (nonassociative) algebra $\cA$ and a Jordan triple system $\cT$, we will denote $\cA_\cR := \cA \otimes \cR$ and $\cT_\cR := \cT \otimes \cR$, with products extended by $\cR$-linearity. The corresponding automorphism group schemes are defined by $\bAut(\cA)(\cR) := \Aut_\cR(\cA_\cR)$ and $\bAut(\cT)(\cR) := \Aut_\cR(\cT_\cR)$. Note that given a Jordan algebra $\cJ$, we can identify $\bAut(\cJ) \leq \bAut(\cT_\cJ) \leq \bAut(\cV_\cJ)$.

\medskip

Let $\cR \in\AlgF$. For $n\in\NN$, the affine group scheme $\bmu_n$ of the $n$-th roots of unity is defined by
$$\bmu_n(\cR) := \{ r\in \cR \med r^n = 1 \}.$$
The multiplicative group scheme is defined by $\bG_m := \bGL_1$, where $\bGL_n$ is the group scheme of $n \times n$ invertible matrices.
Since $\chr\FF\neq 2$, we have that $\bmu_2$ is isomorphic to the constant group scheme $\ZZ/2\ZZ$ (see \cite[Chap.2, Ex.4]{W79}).

\smallskip

Let $N \colon V \to \FF$ be a polynomial map on a vector space $V$, referred to as a \emph{norm}. Then, a \emph{similitude} (or \emph{similarity}) of the norm $N$ is a map $f \in \GL(V)$ such that there exists $\lambda \in\FF^\times$ satisfying $N(f(v)) = \lambda N(v)$ for all $v\in V$. The scalar $\lambda \in\FF^\times$ is called the \emph{multiplier} of the similitude $f$. The group of similitudes of $N$ is denoted by $\Sim(V,N)$. The similitudes with multiplier $1$ are called \emph{isometries}, and form a group denoted by $\Iso(V,N)$. The affine group schemes $\bSim\big(V,N)$ and $\bIso(V,N)$ are defined similarly.

Following the notation in \cite{KMRT98}, if we have a nondegenerate quadratic form $q$ on $V$, with associated symmetric bilinear form $b$, then $\Sim(V,q)$ is called the \emph{general orthogonal group} of $q$ (or of $b$), and is denoted by $\GO(V,q)$ or $\GO(V,b)$; note also that $\Iso(V,q)$ is just the orthogonal group $\Ort(V,q) = \Ort(V,b)$. The affine group schemes $\bGO(V,b)$ and $\bOrt(V,b)$ are defined similarly.

\medskip

For purposes of notation, we will now recall the definition of a central product of groups (see \cite[Chap.2, p.29]{G80}). Let $\gamma_i \colon H \to H_i$ be group isomorphisms with $H_i \leq Z(G_i)$ for $i = 1,2$. Let $N_\gamma := \{ (\gamma_1(h), \gamma_2(h)^{-1}) \med h \in H \} = \{ (h, \gamma(h)^{-1}) \med h \in H_1 \}$, where $\gamma := \gamma_2 \circ \gamma_1^{-1} \colon H_1 \to H_2$. Then the (outer) \emph{central product} of $G_1$ and $G_2$ relative to $\gamma$ is the group $(G_1 \times G_2) / N_\gamma$, which is just ``gluing'' the two copies of $H$. If there is no ambiguity with the isomorphisms, we will denote it by $G_1 \otimes_H G_2$.

\medskip

For further use, we recall the following:

\begin{proposition}[{\cite[Th.~2.5]{A17}} and its proof] \label{schemesJandJTS}
Let $\FF$ be a field of characteristic different from $2$. Let $J$ be a finite-dimensional central simple Jordan $\FF$-algebra with associated Jordan triple system $\cT$. Then, there is an isomorphism of affine group schemes $\bAut(\cT) \simeq \bAut(J) \times \bmu_2$. Here, $\Aut_\cR(\cT_\cR) = \{ r\varphi \med \varphi \in \Aut_\cR(J_\cR), r \in \bmu_2(\cR) \}$.
\end{proposition}

\section{Special Jordan systems of type IV} \label{section.type.IV}
\subsection{Preliminaries for simple Jordan systems of type IV}

\smallskip

\begin{notation} \label{notationType4}
Let $n \in \NN$. Let $W$ be a vector space with $\dim W = n$, and $b \colon W \times W \to \FF$ a nondegenerate symmetric bilinear form. (If the field $\FF$ is algebraically closed, we can assume without loss of generality that $b$ is the standard scalar product on the canonical basis of $\FF^n$.) Consider the simple Jordan pair of type $\text{IV}_n$ given by $\cV = \VIV(W, b) := (W, W)$, and triple products $\cV^\sigma \times \cV^{-\sigma} \times \cV^\sigma \to \cV^\sigma$,
\begin{equation} \label{triple.product.IV}
\{ x,y,z \} := b(x,y)z + b(z,y)x - b(x,z)y.
\end{equation}
The generic trace of $\VIV(W, b)$ is $t(x,y) := b(x,y)$, and the associated quadratic form is $q(x) := \frac{1}{2}b(x,x)$, so that $b(x,y) = q(x+y) - q(x) - q(y)$.
We will also consider the Jordan triple system given by $\cT = \ThIV(W, b) := W$, with the triple product $\cT \times \cT \times \cT \to \cT$ as in Equation~\eqref{triple.product.IV}.

\smallskip

Let $b \colon V \times V \to \FF$ be a nondegenerate symmetric bilinear form  on a vector space $V$ with $\dim V = n-1$. Let $\cJ(V, b) := \FF 1 \oplus V$ be the $n$-dimensional special simple Jordan algebra of bilinear form associated to $(V, b)$. Its product is determined by $1x = x = x1$ and $uv = b(u,v)1$ for $u,v\in V$. If $n = 1$, we just have $\cJ(V, b) := \FF$. We will denote by $\TIV(V, b)$ the $n$-dimensional Jordan triple system associated to $\cJ(V, b)$. 

In the case that $b$ is the standard scalar product on $W = \FF^n$ or $V = \FF^{n-1}$, we may denote $\VIV_n := \VIV(W, b)$, $\ThIV_n := \ThIV(W, b)$, $\TIV_n := \TIV(V, b)$.
\end{notation}

\begin{proposition} \label{isomAlgsBilinearForm}
Let $J = \cJ(V, b) = \FF 1 \oplus V$ be the $n$-dimensional special simple Jordan algebra of bilinear form associated to $(V, b)$. Let $\cT_J =: \TIV(V, b)$ and $\cV_J$ be the Jordan triple system and the Jordan pair associated to $J$, respectively. Extend $b$ to a symmetric bilinear form $\tilde{b}$ on $J$ such that $\tilde{b}(1,1) = 1$ and $\tilde{b}(1,V) = 0$.
Assume that there exists $\bi\in\FF$ with $\bi^2 = -1$. Then 
$$ \cV_J = \cV_\cT \cong \VIV(J, \tilde{b}), $$
where $\cV_\cT$ is the Jordan pair associated to the Jordan triple system $\cT = \cT_J$.
\end{proposition}
\begin{proof}
Denote $\cV = \VIV(J, \tilde{b})$ and let $\cB_V = \{ v_i \}_{i=1}^{n-1}$ be a basis of $V$.
Consider the basis of $\cV^\sigma$ given by $\cB^\sigma = \cB_V \cup \{ 1 \}$, for $\sigma = \pm$.
The triple products of $\cV$ are given by
\begin{align*}
& \{ v_i, v_j, v_k \} = b(v_i,v_j) v_k + b(v_k,v_j) v_i - b(v_i,v_k) v_j,
\quad \{ 1, 1, 1 \} = 1, \\
& \{ 1, v_i, v_j \} = - \{ v_i, 1, v_j \} = b(v_i,v_j) 1,
\quad \{ 1, 1, v_i\} = - \{ 1, v_i, 1 \} = v_i.
\end{align*}
Now, consider the basis of $\cV_J^\sigma$ given by $\cB^\sigma_J = \cB_V \cup \{ \sigma\bi 1 \}$, for $\sigma = \pm$. The triple products $\{x,y,z\} = (xy)z + (zy)x - (zx)y$ of $\cV_J$ are given by
\begin{align*}
& \{ v_i, v_j, v_k \} = b(v_i,v_j) v_k + b(v_k,v_j) v_i - b(v_i,v_k) v_j,
\quad \{ \sigma \bi 1, -\sigma \bi 1, \sigma \bi 1 \} = \sigma \bi 1, \\
& \{ \sigma \bi 1, v_i, v_j \} = - \{ v_i, - \sigma \bi 1, v_j \} = \sigma\bi b(v_i,v_j) 1,
\quad \{ \sigma \bi 1, - \sigma \bi 1, v_i\} = - \{ \sigma \bi 1, v_i, \sigma \bi 1 \} = v_i.
\end{align*}
Then it is clear that we have an isomorphism $\Lambda = (\Lambda^+, \Lambda^-)$ of Jordan pairs, given by
\begin{equation} \label{eqIsomAlgsBilinearForm}
\Lambda^\sigma \colon \cV_J^\sigma \longrightarrow \cV^\sigma,
\quad v_i \longmapsto v_i, \quad \sigma \bi 1 \longmapsto 1.
\end{equation}
\end{proof}

\subsection{Automorphism group schemes of simple Jordan systems of type IV}

\begin{proposition} \label{AutV.IV}
Let $W$ be a vector space with $\dim W = n \in \NN$. Let $b \colon W \times W \to \FF$ be a nondegenerate symmetric bilinear form on $W$. Then
\begin{equation}
\bAut\big(\VIV(W, b)\big) \simeq \bGO(W, b).
\end{equation}
\end{proposition}
\begin{proof}
Denote $\cV = \VIV(W, b)$ and take $\cR \in \AlgF$. Let $\varphi \in \bAut(\cV)(\cR) = \Aut_\cR(\cV_\cR)$. If $\FF$ is algebraically closed, then by \cite[16.7]{L75} the generic minimal polynomial $m(T,X,Y)$ is $\bAut(\cV)$-invariant (that is, we have $m(T,g^+(X),g^-(Y)) = m(T,X,Y)$ for any $g = (g^+,g^-) \in \Aut_\cR(\cV_\cR)$), and consequently the generic trace of $\cV$ (i.e., $t = b$) is also $\bAut(\cV)$-invariant (i.e., $b(g^\sigma(x),g^-(y)) = b(x,y)$ for any $g = (g^+,g^-) \in \Aut_\cR(\cV_\cR)$, $x \in \cV_\cR^\sigma$, $y \in \cV_\cR^{-\sigma}$), that is $\bAut(\cV) = \bAut(\cV, t)$; it is clear that this also holds if $\FF$ is not algebraically closed (which follows by extending scalars). Thus $b(\varphi^\sigma(x), \varphi^{-\sigma}(y)) = b(x,y)$ for any $x,y \in W_\cR$. For $x,y \in W_\cR$, we have $Q_x(y) = b(x,y)x - q(x)y$. Thus
\begin{align*}
b(x,y) & \varphi^\sigma(x) - q(x) \varphi^\sigma(y)
= \varphi^\sigma(Q_x(y)) = Q_{\varphi^\sigma(x)}(\varphi^{-\sigma}(y)) \\
&= b(\varphi^\sigma(x), \varphi^{-\sigma}(y)) \varphi^\sigma(x)
- q(\varphi^\sigma(x)) \varphi^{-\sigma}(y) \\
&= b(x,y) \varphi^\sigma(x) - q(\varphi^\sigma(x)) \varphi^{-\sigma}(y),
\end{align*}
so that 
\begin{equation}
q(x) \varphi^\sigma(y) = q(\varphi^\sigma(x)) \varphi^{-\sigma}(y)
\end{equation}
for $\sigma = \pm$. It follows that $q(x) \neq 0$ if and only if $q(\varphi^\sigma(x)) \neq 0$. Given $x\in W_\cR$ with $q(x) \in \cR^\times$, we have that
$$ \varphi^\sigma(y) = \lambda^\sigma \varphi^{-\sigma}(y), $$
where $\lambda^\sigma = \frac{q(\varphi^\sigma(x))}{q(x)}$ does not depend on the choice of $x$. It follows that
$$ \varphi^+(y) = \lambda^+ \varphi^-(y) = \lambda^+ \lambda^- \varphi^+(y), $$
so that $\lambda^+ \lambda^- = 1$ and $\lambda^\sigma \in \cR^\times$. For each $x,y \in W_\cR$ we have
$$ b(x,y) = b(\varphi^+(x), \varphi^-(y)) = \lambda^{-\sigma} b(\varphi^\sigma(x), \varphi^\sigma(y)), $$
so that $b(\varphi^\sigma(x), \varphi^\sigma(y)) = \lambda^\sigma b(x,y)$ and $q(\varphi^\sigma(x)) = \lambda^\sigma q(x)$, which shows that $\varphi^\sigma \in \GO_\cR(W_\cR, b)$.

Consider the group homomorphism 
\begin{equation} \label{isomGOschemes}
\Phi_\cR \colon \Aut_\cR(\cV_\cR) \to \GO_\cR(W_\cR, b), \quad \varphi \mapsto \varphi^+,
\end{equation}
which is injective because $\varphi^-$ is the dual inverse of $\varphi^+$ relative to the generic trace $t := b$. Take $\varphi^+ \in \GO_\cR(W_\cR, b)$ and let $\lambda^+ \in \cR^\times$ such that $q(\varphi^+(x)) = \lambda^+ q(x)$ for all $x\in W_\cR$. Set $\lambda^- = (\lambda^+)^{-1}$ and $\varphi^- = \lambda^- \varphi^+$. Then
$$ b(\varphi^+(x), \varphi^-(y)) = \lambda^- b(\varphi^+(x), \varphi^+(y))
= \lambda^- \lambda^+ b(x,y) = b(x,y), $$
that is, $\varphi^+$ and $\varphi^-$ are dual inverses of each other. Besides,
$$ q(\varphi^-(x)) = q(\lambda^- \varphi^+(x)) = (\lambda^-)^2 q(\varphi^+(x))
= (\lambda^-)^2 \lambda^+ q(x) = \lambda^- q(x), $$
so that $\varphi^- \in \GO_\cR(W_\cR, b)$. Furthermore,
\begin{align*}
Q_{\varphi^\sigma(x)} & (\varphi^{-\sigma}(y))
= b(\varphi^\sigma(x), \varphi^{-\sigma}(y)) \varphi^\sigma(x)
- q(\varphi^\sigma(x)) \varphi^{-\sigma}(y) \\
&= b(x,y) \varphi^\sigma(x) - \lambda^\sigma q(x) \varphi^{-\sigma}(y)
= b(x,y) \varphi^\sigma(x) - \lambda^\sigma \lambda^{-\sigma} q(x) \varphi^\sigma(y) \\
&= \varphi^\sigma \Big( b(x,y)x - q(x)y \Big) = \varphi^\sigma (Q_x(y)),
\end{align*}
so that $\varphi = (\varphi^+, \varphi^-) \in \Aut_\cR(\cV_\cR)$ and $\Phi_\cR$ is onto. We have proven that $\Phi_\cR$ is an isomorphism, and this clearly defines an isomorphism $\Phi$.
\end{proof}

\begin{corollary} \label{AutT.IV}
Let $W$ be a vector space with $\dim W = n \in \NN$. Let $b \colon W \times W \to \FF$ be a nondegenerate symmetric bilinear form on $W$. Then
$$ \bAut\big(\ThIV(W, b)\big) = \bOrt(W, b). $$
\end{corollary}
\begin{proof}
Denote $\cV = \VIV(W, b)$ and $\cT = \ThIV(W, b)$. We have that $\bAut(\cT) \leq \bAut(\cV)$, where each element $\varphi \in \Aut_\cR(\cT_\cR)$ is identified with $(\varphi, \varphi) \in \Aut_\cR(\cV_\cR)$. The isomorphism in Equation~\eqref{isomGOschemes} shows that $\bOrt(W, b) \lesssim \bAut(\cV)$. Let $\varphi = (\varphi^+, \varphi^-) \in \Aut_\cR(\cV_\cR)$ and $\phi = \varphi^+$. We know that $\varphi^+$ and $\varphi^-$ are dual inverses of each other (where the duality is relative to the generic trace $t$, which is the polar form of $q$). The following statements are equivalent: $\varphi \in \Aut_\cR(\cT_\cR) \iff \varphi^+ = \varphi^- \iff \phi = (\phi^*)^{-1} \iff b(\phi(x), \phi(y)) = b(x,y) \; \forall x,y \in W_\cR \iff \phi\in \Ort_\cR(W_\cR, b) = \bOrt(W, b)(\cR)$, and the result follows.
\end{proof}

\begin{corollary} \label{Aut.TJI}
Let $n \in \NN$ and $\cJ(V, b) = \FF1 \oplus V$ be the $n$-dimensional Jordan algebra of a nondegenerate symmetric bilinear form $b$ on a vector space $V$. Then
\begin{align}
& \bAut\big( \TIV(V, b) \big) \simeq \bOrt(V, b) \times \bmu_2, \\
& \bAut\big( \cJ(V, b) \big) \simeq \bOrt(V, b).
\end{align}
\end{corollary}
\begin{proof}
Let $J = \cJ(V, b)$ and $\cT = \TIV(V, b)$. Then by Prop.~\ref{schemesJandJTS} we have $\bAut(\cT) \simeq \bAut(J) \times \bmu_2$, and it remains to prove the second isomorphism.

By \cite[Prop.~3.1]{AC21} we have that $\bAut(J) = \bStab_{\bAut(\cT_J)}(1) = \bStab_{\bAut(\cV_J)}(1^+,1^-) =: \bG$. Let $\cV = \VIV(J, \widetilde{b})$ be as in Prop.~\ref{isomAlgsBilinearForm} and $\bG' := \bStab_{\bAut(\cV)}(1^+,1^-)$. By extending the field $\FF$ if necessary, we can consider the isomorphism $\Lambda \colon \cV_J \to \cV$ in Equation~\eqref{eqIsomAlgsBilinearForm}, which induces an isomorphism
$$ \bG(\cR) \longrightarrow \bG'(\cR),
\qquad \varphi \longmapsto \Lambda_\cR \circ \varphi \circ \Lambda^{-1}_\cR. $$
On the other hand, the restriction of isomorphism in Equation~\eqref{isomGOschemes} defines a monomorphism
$$ \bG'(\cR) \longrightarrow \Ort_\cR(V_\cR, b),
\qquad \varphi = (\varphi^+,\varphi^-) \longmapsto (\varphi^+)|_V, $$
which is well-defined, because the similitude $\varphi^+$ has multiplier $1$ since $\varphi^+(1)=1$, and so it fixes the subspace $1^\perp = V$. The composition of both maps defines a monomorphism
\begin{equation} \label{isom.IV.proof.coro2}
\bG(\cR) \longrightarrow \Ort_\cR(V_\cR, b), \qquad
\varphi = (\varphi^+, \varphi^-) \longmapsto (\varphi^+)|_V.
\end{equation}
In general (without a field extension), the map in Equation~\eqref{isom.IV.proof.coro2} still defines a group monomorphism; the surjectivity follows since each $f \in \Ort_\cR(V_\cR, b)$ extends to an element $\varphi \in \Aut_\cR(J_\cR)$ determined by $\varphi(1) = 1$, $\varphi|_V = f$. It is clear that Equation~\eqref{isom.IV.proof.coro2} defines an isomorphism $\bG \simeq \bOrt(V, b)$. Note that the result also holds for the trivial case with $n = 1$, where $\bAut(J) \simeq \mathbf{1}$ (the trivial group scheme) and $\bAut(\cT) \simeq \bmu_2$.
\end{proof}

\section{Special Jordan systems of type I} \label{section.type.I}
\subsection{Preliminaries for simple Jordan systems of type I}

\begin{notation} \label{notationPairsTypeI}
Recall from \cite[Chapter~4]{L75} that Jordan pairs of type $\text{I}_{m,n}$, for $m \leq n \in\NN$, are given by $\VtI_{m,n} := \big( \cM_{m,n}(\FF), \cM_{m,n}(\FF) \big)$ with triple products $\{x,y,z\}^\sigma := x y^\Tr z + z y^\Tr x$ for $\sigma = \pm$. The quadratic products are $Q^\sigma_x(y) := x y^\Tr x$. The generic trace of $\VtI_{m,n}$ is $\tr(x, y) := \text{tr}(xy^\Tr)$ where $\text{tr}$ denotes the matrix trace. We will also consider the Jordan triple systems of type $\text{I}_{m,n}$ defined by $\TtI_{m,n} := \cM_{m,n}(\FF)$ with triple products $\{x,y,z\} := x y^\Tr z + z y^\Tr x$. We will also denote $\VtI_n := \VtI_{n,n}$, $\TtI_n := \TtI_{n,n}$.

Another well-known construction for special simple Jordan pairs of type $\text{I}_{m,n}$, which can be found in \cite{S84}, is given by $\VhI_{m,n} := \big( \cM_{m,n}(\FF), \cM_{n,m}(\FF) \big)$, with triple products $\{x,y,z\}^\sigma := xyz + zyx$. Here, $Q^\sigma_x(y) := xyx$. In the case $m = n$, the Jordan pair $\VhI_n := \VhI_{n,n}$ is associated to the Jordan triple system $\ThI_n := \cM_n(\FF)$, with triple products $\{x,y,z\} := xyz + zyx$. Recall that $\cM_n(\FF)^{(+)} := \cM_n(\FF)$ is a Jordan algebra with the symmetric product $x \circ y := \frac{1}{2}(xy + yx)$. Then the Jordan triple system associated to $\cM_n(\FF)^{(+)}$ is $\ThI_n$, because it has triple products
$\{x,y,z\} := 2 \big( (x \circ y) \circ z + (z \circ y) \circ x - (z \circ x) \circ y \big)
= xyz + zyx$. Consequently, the Jordan pair associated to $\cM_n(\FF)^{(+)}$ is $\VhI_n$.
We will show that the automorphism group schemes of $\ThI_n$ and $\TtI_n$ are not isomorphic, which implies that $\ThI_n \ncong \TtI_n$.

Note that there is an isomorphism $\varphi=(\varphi^+,\varphi^-) \colon \VtI_{m,n} \to \VhI_{m,n}$ given by $\varphi^+(x) = x$ and $\varphi^-(y) = y^\Tr$. Therefore, the generic trace of $\VhI_{m,n}$ is $\tr(x, y) := \text{tr}(xy)$, and we get an isomorphism $\bAut(\VtI_{m,n}) \simeq \bAut(\VhI_{m,n})$ given by 
\begin{equation} \label{AutVIisomorphism}
\Aut_\cR \big( (\VtI_{m,n})_\cR \big) \longrightarrow \Aut_\cR \big( (\VhI_{m,n})_\cR \big), \qquad
f = (f^+, f^-) \longmapsto (f^+, \widetilde{\sigma}_\cR(f^-)),
\end{equation}
where
\begin{equation} \label{sigmaIsomorphism}
\widetilde{\sigma}_\cR \colon \GL_\cR(\cM_{m,n}(\cR)) \longrightarrow \GL_\cR(\cM_{n,m}(\cR)), \qquad
\widetilde{\sigma}_\cR(\psi)(X) := \psi(X^\Tr)^\Tr.
\end{equation}
\end{notation}

\begin{remark} \label{block.decomposition}
Let $m,n\in\NN$ and $k = m+n$. Consider the Lie algebra $\cL = \cM_k(\FF)^{(-)} := \cM_k(\FF)$ with the product $[x,y] := xy-yx$. Then the block decomposition
\begin{equation*}
\cL \equiv \left(\begin{array}{c|c} \cM_m(\FF) & \cM_{m,n}(\FF) \\
\hline \cM_{n,m}(\FF) & \cM_n(\FF) \end{array}\right)
\end{equation*}
defines a $\ZZ$-grading $\cL = \bigoplus_{i=-1}^{1} \cL_i$, where $\cL_0 \equiv \cM_m(\FF) \oplus \cM_n(\FF)$ corresponds to the diagonal blocks, and where we identify $\cL_1 \equiv \cM_{m,n}(\FF)$, $\cL_{-1} \equiv \cM_{n,m}(\FF)$. Recall that any $\ZZ$-graded Lie algebra $\cL = \bigoplus_{i=-1}^{1} \cL_i$ defines a Jordan pair $\cV_\cL = (\cV^+, \cV^-) := (\cL_1, \cL_{-1})$ with triple products given by $\{x,y,z\} := [[x,y],z]$. In our case, it is easy to see that $\cV_\cL \cong \VhI_{m,n} = \big( \cM_{m,n}(\FF), \cM_{n,m}(\FF) \big)$.
\end{remark}

\begin{notation}
Let $L_a$ and $R_a$ denote, respectively, the left and right multiplications by an element $a$. For each $a\in\GL_m(\FF)$ and $b\in\GL_n(\FF)$, we will also consider the automorphisms
$\widetilde{L}_a, \widetilde{R}_b \in \Aut_\cR\big( (\VtI_{m,n})_\cR \big)$
defined by
\begin{equation} \label{generators.V.I}
\widetilde{L}_a := (L_a, L_{a^\Tr}^{-1}), \qquad \widetilde{R}_b := (R_b, R_{b^\Tr}^{-1}).
\end{equation}
Through the isomorphism in Equation~\eqref{AutVIisomorphism}, these correspond to $\widehat{L}_a, \widehat{R}_b \in \Aut_\cR\big( (\VhI_{m,n})_\cR \big)$ defined by
\begin{equation} \label{generators.V.I.bis}
\widehat{L}_a := (L_a, R_a^{-1}), \qquad \widehat{R}_b := (R_b, L_b^{-1}).
\end{equation}
\end{notation}

\begin{notation} \label{notationBarSchemes}
We will now follow \cite[\S 3.1]{EKmon}, with some notation changes. 
Define the affine group scheme $\overline{\bAut}^{(+)}(\cM_n(\FF))$, where
$\overline{\bAut}^{(+)}(\cM_n(\FF))(\cR) := \overline{\Aut}^{(+)}_\cR(\cM_n(\cR))$
is the set consisting of all the maps $\cM_n(\cR) \to \cM_n(\cR)$ of the form
\begin{equation} \label{barSchemePlus}
X \mapsto e_1 \psi(X) + e_2 \psi(X^\Tr),
\end{equation}
where $e_1 \in \cR$ is an idempotent, $e_2 = 1 - e_1$, and $\psi \in \Aut_\cR(\cM_n(\cR))$.
Equivalently, the elements $\varphi \in \overline{\Aut}^{(+)}_\cR(\cM_n(\cR))$ are the $\cR$-linear bijections $\varphi \colon \cM_n(\cR) \to \cM_n(\cR)$ such that there is a direct product decomposition $\cR = \cR_1 \times \cR_2$, where $\cR_i = e_i \cR$ for some idempotents $e_1,e_2 \in \cR$ with $e_1 + e_2 = 1$, and such that $\varphi$ restricts to an automorphism of $\cM_n(\cR_1)$ and an antiautomorphism of $\cM_n(\cR_2)$.

Similarly, we can define the affine group scheme $\overline{\bAut}^{(-)}(\cM_n(\FF))$, where we use the maps
\begin{equation} \label{barSchemeMinus}
X \mapsto e_1 \psi(X) - e_2 \psi(X^\Tr)
\end{equation}
instead of Equation~\eqref{barSchemePlus}. Note that $\overline{\bAut}^{(-)}(\cM_n(\FF))$ was denoted by $\overline{\bAut}(\cM_n(\FF))$ in \cite[\S 3.1]{EKmon}. We claim that
\begin{equation} \label{barSchemeIsom}
\overline{\bAut}^{(+)}(\cM_n(\FF))
\simeq \bAut(\cM_n(\FF)) \rtimes \bmu_2
\simeq \overline{\bAut}^{(-)}(\cM_n(\FF)).
\end{equation}
The proof of the second isomorphism in Equation~\eqref{barSchemeIsom} is known (see \cite[Equation~(3.19)]{EKmon}), and now we will mimic its proof in order to prove the first isomorphism.

Recall that $\bmu_2(\cR)$ can be identified with the set of idempotents of $\cR$ through the map $\tau \mapsto e_1 = \frac{1}{2}(1+\tau)$, and the product of $\bmu_2(\cR)$ is transferred to a product of idempotents of $\cR$ given by
$$ e' * e'' := e'e'' + (1 - e')(1 - e'').$$
Then, the product of the group $\Aut_\cR(\cM_n(\cR)) \rtimes \bmu_2(\cR)$ is given by
$$(\psi', e')(\psi'', e'') = \big( \psi'(e'\psi'' + (1-e')\widetilde{\sigma}_\cR(\psi'')), e'*e'' \big),$$
where $\widetilde{\sigma}_\cR \colon \Aut_\cR(\cM_n(\cR)) \to \Aut_\cR(\cM_n(\cR))$
is defined as in Equation~\eqref{sigmaIsomorphism}.

Let
$X \mapsto e'_1 \psi'(X) + e'_2 \psi'(X^\Tr)$ and
$X \mapsto e''_1 \psi''(X) + e''_2 \psi''(X^\Tr)$
be maps of $\overline{\Aut}^{(+)}_\cR(\cM_n(\cR))$.
Their composition is in $\overline{\Aut}^{(+)}_\cR(\cM_n(\cR))$ because it is given by
\begin{align*}
e'_1 \psi' & \Big( e''_1 \psi''(X) + e''_2 \psi''(X^\Tr) \Big)
+ e'_2 \psi' \Big( e''_1 (\psi''(X))^\Tr + e''_2 (\psi''(X^\Tr))^\Tr \Big) \\
&= \Big( e'_1 e''_1 \psi' \psi'' + e'_2 e''_2 \psi' \widetilde{\sigma}_\cR(\psi'') \Big) (X)
+ \Big( e'_1 e''_2 \psi' \psi'' + e'_2 e''_1 \psi' \widetilde{\sigma}_\cR(\psi'') \Big) (X^\Tr) \\
&= e_1 \psi(X) + e_2 \psi(X^\Tr),
\end{align*}
where $e_1 = e'_1 e''_1 + e'_2 e''_2 = e'_1 * e''_1$, $e_2 = e'_1 e''_2 + e'_2 e''_1 = 1 - e_1$,
$\psi = e'_1 \psi' \psi'' + e'_2 \psi' \widetilde{\sigma}_\cR(\psi'')$.
It follows that an isomorphism 
$\bAut(\cM_n(\FF)) \rtimes \bmu_2 \simeq \overline{\bAut}^{(+)}(\cM_n(\FF))$
is given by
\begin{equation*}
\Aut_\cR(\cM_n(\cR)) \rtimes \bmu_2(\cR) \to  \overline{\Aut}_\cR^{(+)}(\cM_n(\cR)), 
\qquad (\psi, e_1) \mapsto \big( X \mapsto e_1 \psi(X) + (1-e_1) \psi(X^\Tr) \big).
\end{equation*}
For each $\sigma = \pm$, let $\bmu_2^{(\sigma)}$ denote the corresponding copy of $\bmu_2$ in $\overline{\bAut}^{(\sigma)}(\cM_n(\FF))$, so we get
\begin{equation}
\overline{\bAut}^{(\sigma)}(\cM_n(\FF)) = \bAut(\cM_n(\FF)) \rtimes \bmu_2^{(\sigma)} \leq \bGL_n(\cM_n(\FF)).
\end{equation}
\end{notation}

\begin{remarks} \label{remarkSchemes} \ \\
\noindent \textbf{1)} Consider the morphism $\text{Ad} \colon \bGL_n \to \bGL(\cM_n(\FF))$ given by the adjoint representation, i.e., $\text{Ad}_\cR(g) \colon x \mapsto gxg^{-1}$ for $g\in\GL_n(\cR)$, $x\in\cM_n(\cR)$. As shown in \cite[3.1]{EKmon}, the image of this morphism is $\bAut(\cM_n(\FF))$, i.e., it factors through a quotient map $\text{Ad} \colon \bGL_n \to \bAut(\cM_n(\FF))$. Hence, due to the sheaf property of quotient maps \cite[Theorem 15.5]{W79}, for every $\cR \in \AlgF$ and every $\varphi$ in $\bAut(\cM_n(\FF))(\cR)$, there is a faithfully flat extension $f\colon \cR\to \cS$ and $x\in \bGL_n(\cS)$ such that $\text{Ad}_\cS(x)=\bAut(\cM_n(\FF))(f)(\varphi)$. From \cite[Rem.~3.5 and Ex.~A.27]{EKmon} we know that the kernel of $\text{Ad}$ is $\bG_m$ and $\bPGL_n := \bGL_n / \bG_m \cong \bAut(\cM_n(\FF))$ via the adjoint representation.

\noindent \textbf{2)} Our goal now is to prove the isomorphism in Equation~\eqref{isom.Iso.det}; we also need to recall the isomorphism in Equation~\eqref{eqDetSimCentralProduct}, to be used in subsequent results.

From \cite[Theorem 3.4]{W87} and \cite[Corollary 1.4.2]{W87}, we have
\begin{equation} \label{eqDetSimCentralProduct}
	\bSim\big(\cM_n(\FF), \det \big)
	\simeq ( \bGL^2_n / \bT ) \rtimes \bmu_2
	= (\bGL_n \otimes_{\bG_m} \bGL_n ) \rtimes \bmu_2,
\end{equation}
where
\begin{equation*}
	\bT(\cR) := \{ (r1, r^{-1}1) \med r\in \cR^\times \} \cong \cR^\times = \bG_m(\cR);
\end{equation*}
here $\bmu_2$ acts by swapping of the components of $\bGL_n \otimes_{\bG_m} \bGL_n$, which corresponds to the transpositions $\bmu_2^{(+)}$ in $\bSim\big(\cM_n(\FF), \det \big)$ (defined as in Notation~\ref{notationBarSchemes}).

Moreover, for every $\cR \in \AlgF$ and $a,b \in \bGL_n(\cR)$, since the isomorphism given in the theorem sends $a \otimes b \in (\bGL^2_n/\bT)(\cR)$ to $L_aR_{b^\Tr}$, it follows from the sheaf property of quotient maps that for every $\varphi\in \bSim\big(\cM_n(\FF), \det \big)(\cR)$ there is a faithfully flat extension $f\colon \cR\to \cS$, $a,b\in \bGL_n(\cS)$ and $\tau\in \bmu_2^{(+)}(\cS)$ such that $\bSim\big(\cM_n(\FF), \det \big)(f)(\varphi)=L_aR_b\tau$. If we assume, in addition, that $\varphi\in \bIso\big(\cM_n(\FF), \det \big)(\cR)$, then  due to  \cite[Theorem 4.1]{W87} we can choose $a,b$ such that $\det(ab)=1$. Let $\bG$ denote the subgroup scheme of $\bGL_n^2$ whose $\cR$-points are
$$ \bG(\cR)=\{(a,b)\in \bGL^2_n(\cR) \mid \det(ab)=1\}, $$
which is well-defined because it is a fiber product of group schemes;
then, through isomorphism in Equation~\eqref{eqDetSimCentralProduct} we get an isomorphism
\begin{equation}
\varphi_1 \colon (\bG/\bT) \rtimes \bmu_2 \longrightarrow \bIso\big(\cM_n(\FF), \det \big).
\end{equation}

Consider the inclusion
\begin{equation}
\iota\colon \bSL^2_n/\bmu_n = \bSL_n\otimes_{\bmu_n} \bSL_n \longrightarrow \bG/\bT,
\end{equation}
where we identify $\bmu_n$ with the kernel of the morphism $\bSL_n \times \bSL_n \to \bG/\bT$ given by $(a,b) \mapsto a \otimes b$. Then $\iota$ is clearly a morphism, and to show that it is a quotient map, take $x\in (\bG/\bT)(\cR)$. Due to the sheaf property of quotient maps there is a faithfully flat extension $\cR\to \cS$ and $a,b\in\bGL_n(\cS)$ such that $\det(ab)=1$ satisfying that the image of $x$ in $(\bG/\bT)(\cS)$ is $a \otimes b$. Let $\cT=\cS[Y]/\langle Y^n-\det(a)\rangle$, and let $y$ be the class of $Y$ on the quotient $\cT$. Since $\cT$ is a free $\cS$-module, it is a faithfully flat extension of $\cS$ and due to \cite[Theorem 13.3]{W79}, a faithfully flat extension of $\cR$. Since the image of $a \otimes b$ (and thus, the image of $x$) in $(\bG/\bT)(\cT)$ is $\iota(y^{-1}a \otimes yb)$, then the sheaf property of quotient maps is satisfied and therefore $\iota$ is a quotient map. Note that $\iota$ extends to an isomorphism
\begin{equation}
\varphi_2 \colon (\bSL_n\otimes_{\bmu_n} \bSL_n) \rtimes \bmu_2 \longrightarrow (\bG/\bT) \rtimes \bmu_2,
\end{equation}
where $\bmu_2$ acts by swapping of the components. By composition of $\varphi_1$ and $\varphi_2$ we get an isomorphism
\begin{equation} \label{isom.Iso.det}
(\bSL_n \otimes_{\bmu_n} \bSL_n) \rtimes \bmu_2\simeq\bIso\big( \cM_n(\FF), \det \big),
\end{equation}
which sends $a \otimes b$ to $L_aR_{b^\Tr}$, and $\bmu_2$ to the corresponding transpositions in $\bIso\big( \cM_n(\FF), \det \big)$.

\noindent \textbf{3)} Let $n\geq 2$ and consider the Jordan algebra $J = \cM_n(\FF)^{(+)}$. We claim that 
\begin{equation} \label{structureMnPlus}
\bAut(J) = \overline{\bAut}^{(+)}\big(\cM_n(\FF)\big).
\end{equation}

First consider the case with $n\geq 3$. Then by \cite[Th.~5.47]{EKmon} and its proof we know the following facts.
We know that the restriction map $\theta \colon \bAut(U(J), *) \to \bAut(J)$ is an isomorphism of affine group schemes, where we consider the inclusion $\iota \colon J \to U(J) = \cM_n(\FF) \times \cM_n(\FF)^\text{op}$, $X \mapsto (X,X)$, and where $*$ denotes the exchange involution $(X,Y) \mapsto (Y,X)$.
We also know that each $\varphi \in \Aut_\cR(U(J)_\cR, *)$
is of the form
$$ \varphi(e_1X_1 + e_2X_2, e_1Y_1 + e_2Y_2)
= (e_1\psi(X_1) + e_2\psi(Y_2^\Tr), e_1\psi(Y_1) + e_2\psi(X_2^\Tr)),$$
where $e_1 \in \cR$ is an idempotent, $e_2 = 1 - e_1$, and $\psi \in \Aut_\cR(\cM_n(\cR))$.
Therefore, if we apply the isomorphism $\theta$, we get the equality in Equation~\eqref{structureMnPlus}.

Now consider the case with $n = 2$. By the Cayley-Hamilton theorem, for each $x \in \cM_2(\FF)$ we have $x^2 - \tr(x)x + \det(x)1 = 0$. Let $V$ be the traceless subspace of $\cM_2(\FF)$. Thus we have $x^2 = -\det(x)1$ for each $x \in V$, and note that $q(x) := -\det(x)$ is a quadratic form on $V$. Thus we can identify $\cM_2(\FF)^{(+)} = \cJ(V, b)$ where $b$ is obtained by scaling the polar form of $q$. Using Corollary~\ref{Aut.TJI} we see that
$\bOrt(V, b) \simeq \bAut(J) \leq \bIso\big( \cM_2(\FF), \det \big)$.
Consider the morphism $\theta\colon \bAut(\cM_2(\FF))\rtimes \bmu_2^{(+)}\to \bAut(J)$ defined in the obvious way, which defines a monomorphism for each $\cR \in \AlgF$, so $\theta$ is a closed imbedding. In order to prove that $\theta$ is a quotient map, we take $\cR \in \AlgF$ and $\varphi\in \bAut(J)(\cR) \leq \bIso\big( \cM_2(\FF), \det \big)(\cR)$. Due to Remark~\ref{remarkSchemes}-2), there is a faithfully flat extension $f\colon \cR\to \cS$, $a,b\in \bGL_n(\cS)$ with $\det(ab)=1$ and $\tau\in \bmu_2^{(+)}(\cS)$ satisfying that $\bAut(J)(f)(\varphi)=L_aR_b\tau$. Since the automorphism $\varphi$ must preserve the unit, we get $b=a^{-1}$. Hence, $\bAut(J)(f)(\varphi) = \theta_\cS \big( (\Ad_\cS(a),\tau) \big)$, so the sheaf property of quotient maps is satisfied and therefore $\theta$ is a quotient map.
\end{remarks}

\subsection{Automorphism group schemes of simple Jordan systems of type I}

\begin{proposition} \label{automorphismgroupscheme.I.square}
Let $1 < n\in\NN$. Then there is an isomorphism of affine group schemes
\begin{equation}
\bAut(\VhI_n) \simeq \bSim\big(\cM_n(\FF), \det \big)
\simeq (\bGL_n \otimes_{\bG_m} \bGL_n ) \rtimes \bmu_2,
\end{equation}
where $\bmu_2$ acts by swapping of the components of $\bGL_n \otimes_{\bG_m} \bGL_n$. \\
Furthermore, we have an isomorphism
\begin{equation}
\theta\colon  (\bGL_n \otimes_{\bG_m} \bGL_n ) \rtimes \bmu_2^{(+)} \longrightarrow \bAut(\VhI_n)
\end{equation}
determined by
\begin{equation}
\theta_\cR \big( (a \otimes b,\tau) \big) := \widehat{L}_a\widehat{R}_{b^\Tr}\tau
\end{equation}
for every $\cR \in \AlgF$, $a,b\in\bGL_n(\cR)$ and $\tau\in\bmu_2^{(+)}(\cR)$, where $\bmu_2^{(+)}$ acts by swapping on $\bGL_n \otimes_{\bG_m} \bGL_n$.

$\bullet$ For the trivial case we have $\bAut(\VhI_1) \simeq \bG_m$, and
$\Aut_\cR\big((\VhI_1)_\cR\big) = \{ \widehat{L}_a \med a\in \cR^\times \}$.
\end{proposition}
\begin{proof}
We will omit the proof for the trivial case. Let $n > 1$ and $\cV = \VhI_n$.
It is easy to see that $\theta$ is a well-defined morphism, and a closed imbedding. To show that $\theta$ is a quotient map, take $\varphi\in\bAut(\cV)(\cR) = \Aut_\cR(\cV_\cR)$. Let $a := \varphi^+(1)$, $b := \varphi^-(1)$. We claim that $a,b\in\GL_n(\cR)$ and $a^{-1} = b$. Since $D_{1,1} = \id$, we also have that
$D_{a,b} = D_{\varphi^+(1), \varphi^-(1)} = \varphi^+ \circ D_{1,1} \circ (\varphi^+)^{-1} = \id$.
Then, for each $x\in\cM_n(\cR)$ we have that $x = D_{a,b}(x) = \frac{1}{2}(abx + xba)$. By taking $x = E_{ii}$, it follows that $ab$ and $ba$ are diagonal, so we can write $ab = \diag(\lambda_1,\hdots,\lambda_n)$ and $ba = \diag(\mu_1,\hdots,\mu_n)$, and then by taking $x = E_{ij}$ it follows that $\mu_i+\lambda_j = 2$ for each $1\leq i,j \leq n$, thus $ab = \lambda 1$, $ba = \mu 1$ with $\lambda + \mu = 2$ for some $\lambda, \mu \in \cR$. Note that for each $0 \neq r \in \cR$, we have $0\neq\varphi^+(r1) = r \varphi^+(1) = ra$. Since $\lambda a = (ab)a = a(ba) = \mu a$, we get $ra = 0$ with $r = \lambda - \mu \in \cR$, which implies that $r = 0$, and it follows that $\lambda = \mu = 1$. Therefore $ab = 1 = ba$ and the claim follows.

The composition $\psi = \widehat{L}^{-1}_a \varphi$ fixes $1^+$ and $1^-$. Recall that $\cV = \VhI_n$ is the Jordan pair associated to the Jordan algebra $J = \cM_n(\FF)^{(+)}$. By \cite[Prop.~3.1]{AC21} we know that the stabilizer of $(1^+,1^-)$ in $\bAut(\cV)$ is $\bAut(J)$. Therefore, $\psi^+ = \psi^- \in \Aut_\cR(J_\cR)$. Due to Remark \ref{remarkSchemes}-3, we know that there is a faithfully flat extension $f\colon \cR\to \cS$, an element $b\in \bGL_n(\cS)$ and $\tau\in\bmu_2^{(+)}(\cS)$ such that $\bAut(J)(f)(\psi^+)=L_bR_b^{-1}\tau$. Thus, $\bAut(\cV)(f)(\psi)=\widehat{L}_b\widehat{R}_b^{-1}\tau \in \im(\theta_\cS)$. It is clear that $\bAut(\cV)(f)(\widehat{L}_a) = \widehat{L}_a = \theta_\cS \big( (a\otimes1,1) \big) \in \im(\theta_\cS)$. Since $\theta$ is a natural transformation, it follows that $\bAut(\cV)(f)(\varphi) = \bAut(\cV)(f)(\widehat{L}_a) \bAut(\cV)(f)(\psi)$. Hence, $\bAut(\cV)(f)(\varphi)\in \im(\theta_\cS)$, so the sheaf property of quotient maps is satisfied, and therefore $\theta$ is a quotient map.
\end{proof}

\begin{proposition} \label{automorphismgroupscheme.I.rectangle}
Let $m,n \in \NN$ with $m < n$. Then there is an isomorphism of affine group schemes
\begin{equation}
\gamma\colon(\bGL_m \times \bGL_n) / \bT
= \bGL_m \otimes_{\bG_m} \bGL_n \longrightarrow \bAut(\VhI_{m,n}),
\end{equation}
given by $\gamma(a \otimes b)=\widehat{L}_a\widehat{R}_{b^\Tr}$, where $\otimes_{\bG_m}$ denotes a central product relative to $\bG_m$, and
\begin{equation*}
\bT(\cR) := \{ (r1, r^{-1}1) \med r\in \cR^\times \} \cong \cR^\times = \bG_m(\cR).
\end{equation*}
\end{proposition}
\begin{proof}

\noindent $\blacksquare$ $\boxed{\text{Case 1}}$. Let $k = m + n \geq 3$. If $k = 3$, we will also assume that $\chr \FF \neq 3$. (For the case with $k = 3 = \chr \FF$, since the Lie algebra $\mathfrak{a}_2 = \psl_3(\FF)$ is exceptional, a different approach will be used.)

Set $\cV = \VhI_{m,n}$. Note that $\psl_k(\FF) = [\cM,\cM] / (Z(\cM) \cap [\cM,\cM])$ with $\cM = \cM_k(\FF)$. Let
$$\theta \colon \overline{\bAut}^{(-)}(\cM_k(\FF)) \longrightarrow \bAut(\psl_k(\FF))$$
be defined by restriction and passing to the quotient modulo the center. Then by \cite[Th.3.9]{EKmon}, under our restrictions on $k$ and $\chr\FF$, we know that $\theta$ is an isomorphism of affine group schemes. 

As in Remark~\ref{block.decomposition}, we can identify $\VhI_{m,n}$ with the blocks of a $\ZZ$-grading on $\cM_k(\FF)^{(-)}$; let $\Gamma$ denote this grading. The isomorphism $\theta$ shows that the former grading corresponds to a $\ZZ$-grading on the quotient, $\widetilde{\Gamma} \colon \psl_k(\FF) = \bigoplus_{i=-1}^1 \cL_i$, and it is clear that $\VhI_{m,n}$ and its triple products can also be recovered from the grading $\widetilde{\Gamma}$, where we identify $\cL_1 \equiv \cV^+$, $\cL_{-1} \equiv \cV^-$. Since $\psl_k(\FF)$ and $\cV$ are simple, from the well-known correspondence between simple $\ZZ$-graded Lie algebras and simple Jordan pairs (see \cite{CS11} and references therein) it is clear that $\psl_k(\FF) = \TKK(\cV)$ (the TKK-construction using $\cV$). It is well-known that automorphism group schemes extend through the $\TKK$-construction (e.g., see \cite[Equation~(2.8)]{AC21}), so that we can identify $\bAut(\cV) \leq \bAut\big(\psl_k(\FF)\big)$; conversely, the automorphisms of $\psl_k(\FF)$ which are compatible with $\widetilde{\Gamma}$ can be restricted to $\cV$. Therefore, the elements of $\bAut(\cV)$ can be recovered (by restriction) from the elements of $\bAut\big(\psl_k(\FF)\big)$ which are compatible with $\widetilde{\Gamma}$, or from the elements of $\overline{\bAut}^{(-)}\big(\cM_k(\FF)\big)$ which are compatible with $\Gamma$.

As in the proofs above, it is easy to see that the morphism $\gamma$ is well-defined and a closed imbedding, and we need to show that it is a quotient map.

\medskip

Assume by contradiction that there exists $\psi \in \Aut_\cR(\cM_k(\cR))$ that swaps the subspaces corresponding to $\cV^+_\cR$ and $\cV^-_\cR$. By Remark~\ref{remarkSchemes}-1), we know that there is a faithfully flat extention $f\colon \cR\to \cS$ and some $M \in \GL_k(\cS)$ such that $\bAut(\cM_k(\FF))(f)(\psi)(X) = MXM^{-1}$ for each $X \in \cM_k(\cS)$. Write
$M = \left(\begin{array}{c|c} M_{11} & M_{12} \\ \hline M_{21} & M_{22} \end{array}\right)$
with $M_{11} \in \cM_m(\FF)$ and $M_{22} \in \cM_n(\FF)$.
For each $A_{12} \in \cV^+_\cS$ there exists some $B_{21} \in \cV^-_\cR$ such that
$M A M^{-1} = B$, where
$A = \left(\begin{array}{c|c} 0 & A_{12} \\ \hline 0 & 0 \end{array}\right)$,
$B = \left(\begin{array}{c|c} 0 & 0 \\ \hline B_{21} & 0 \end{array}\right)$.
Since $MA = BM$, we get that $M_{11}A_{12} = 0$ for all $A_{12} \in \cV^+_\cS$, which implies that $M_{11} = 0$.
An analogous argument with $A_{21} \in \cV^-_\cS$ shows that $M_{22} = 0$.
Since $M$ is equivalent to $\widetilde{M} = \diag(M_{12}, M_{21})$ and $\det(\widetilde{M}) = 0$ (because $m < n$), it follows that $M$ is not invertible, a contradiction.

Take $\varphi \in \Aut_\cR(\cV_\cR) \leq \overline{\Aut}_\cR^{(-)}(\cM_k(\cR))$ and write
$\varphi(X) = e_1 \psi(X) - e_2 \psi(X^\Tr)$,
where $e_1 \in \cR$ is an idempotent, $e_2 = 1 - e_1$, and $\psi \in \Aut_\cR(\cM_k(\cR))$.
Let $\cR_i = e_i \cR$. Since $\varphi$ fixes $\cV_\cR^+$ and $\cV_\cR^-$, it follows that $\psi$ swaps the subspaces $\cV_{\cR_2}^+ \leftrightarrow \cV_{\cR_2}^-$ corresponding to $\cV_{\cR_2}$, which is impossible (as we have shown above) unless $\cR_2 = 0$. Thus $e_2 = 0$ and $\varphi = \psi \in \Aut_\cR(\cM_k(\cR))$, there exists a faithfully flat extention $f\colon \cR\to \cS$ and some $M \in \GL_k(\cS)$ such that $\bAut(\cM_k(\FF))(f)(\varphi)(X) = MXM^{-1}$. Again, write
$M = \left(\begin{array}{c|c} M_{11} & M_{12} \\ \hline M_{21} & M_{22} \end{array}\right)$.
For each $A_{12} \in \cV^+_\cS$ there is some $B_{12} \in \cV^+_\cS$ such that
$M A M^{-1} = B$, where
$A = \left(\begin{array}{c|c} 0 & A_{12} \\ \hline 0 & 0 \end{array}\right)$,
$B = \left(\begin{array}{c|c} 0 & B_{12} \\ \hline 0 & 0 \end{array}\right)$.
Since $MA = BM$, we get $M_{21}A_{12} = 0$ for each $A_{12} \in \cV^+_\cR$, thus $M_{21} = 0$, and analogously $M_{12} = 0$. Then $M = \diag(M_{11},M_{22})$ and $\bAut(\cV)(f)(\varphi) = \gamma_{\cS}(M_{1,1} \otimes (M_{2,2}^\Tr)^{-1})$. 
It follows that $\gamma$ is a quotient map because of the sheaf property.

\noindent $\blacksquare$ $\boxed{\text{Case 2}}$. It remains to consider the case with $k = m + n = 3$ and $\chr \FF = 3$ (thus $m = 1$, $n = 2$). Instead, we will consider the more general case with $m = 1 \leq n$.

Let $\cV = \VhI_{1,n}$ and fix $\varphi \in \Aut_\cR(\cV_\cR)$. Since $\varphi^+ \in \GL_\cR(\cV^+_\cR) = \GL_\cR(\cR^n) \cong \GL_n(\cR)$, we can write $\varphi^+ = R_a$ for some $a \in\GL_n(\cR)$. Recall that $\varphi^+$ determines $\varphi^-$ (and viceversa), because the generic trace $t \colon \cV^- \times \cV^+ \to \FF$ is $\bAut(\cV)$-invariant (consequence of \cite[16.7]{L75}). Consequently, $\varphi = \widehat{R}_a$. Therefore, we have an isomorphism $\bGL_n \to \bAut(\cV)$ given by $\GL_n(\cR) \to \Aut_\cR(\cV_\cR)$, $a \mapsto \widehat{R}_{a^\Tr}$. Also note that $\bGL_1 \otimes_{\bG_m} \bGL_n = \bG_m \otimes_{\bG_m} \bGL_n \simeq \bGL_n$. The result follows as in the case above.
\end{proof}

\begin{corollary} \label{Aut.tildeTImn} \;
\begin{itemize}
\item[$1)$] For each $m<n\in \NN$, there is an isomorphism
\begin{equation}
\gamma\colon \bOrt_m\otimes_{\bmu_2} \bOrt_n \longrightarrow \bAut(\TtI_{m,n})
\end{equation}
given by $\gamma(a \otimes b)=\widetilde{L}_a\widetilde{R}_{b^\Tr}$.
\item[$2)$] For each $1<n\in \NN$, we have
\begin{equation}
(\bOrt_n\otimes_{\bmu_2} \bOrt_n) \rtimes \bmu_2 \simeq \bAut(\TtI_{n}),
\end{equation}
where $\bmu_2$ acts by swapping of the coordinates of $\bOrt_n\otimes_{\bmu_2} \bOrt_n$. \\
Furthermore, there is an isomorphism
\begin{equation}
\gamma \colon (\bOrt_n\otimes_{\bmu_2} \bOrt_n) \rtimes \bmu_2^{(+)} \longrightarrow \bAut(\TtI_{n})
\end{equation}
given by $\gamma(a \otimes b, \tau) = \widetilde{L}_a\widetilde{R}_{b^\Tr} \tau$.
\item[$3)$] For the trivial case, $\bAut(\TtI_1) \simeq \bmu_2$, and
$\Aut_\cR\big((\TtI_1)_\cR\big) = \{ L_a \med a \in \bmu_2(\cR) \}$.
\end{itemize}	
\end{corollary}
\begin{proof}
For each similitude $a \in \GO_n(\cR)$, let $m_a \in \cR^\times$ denote its multiplier (so that $a^\Tr a = m_a 1$). In case 1), consider the subgroup scheme $\bG$ of $\bGO_m\times\bGO_n$ whose $\cR$-points are
$$ \bG(\cR)=\{(a,b)\in\bGO_m(\cR)\times\bGO_n(\cR) \mid m_am_b=1\}. $$
Then $\bG$ is well-defined because it is a fiber product of group schemes. Consider the inclusion
\begin{equation}
\iota\colon \bOrt_m\otimes_{\bmu_2} \bOrt_n \longrightarrow \bG/\bT
\end{equation}
where $\bT$ is the same subgroup scheme as in proposition \ref{automorphismgroupscheme.I.rectangle}. Then $\iota$ is clearly a closed imbedding. To show that $\iota$ is a quotient map, take $x\in (\bG/\bT)(\cR)$. Due to the sheaf property of quotient maps there is a faithfully flat extension $\cR\to \cS$ and $a\in\bGO_m(\cR)$, $b\in \bGO_n(\cR)$ such that $m_am_b=1$ satisfying that the image of $x$ in $(\bG/\bT)(\cS)$ is $a \otimes b$. Let $\cT = \cS[Y]/\langle Y^2-m_a\rangle$, and let $y$ be the class of $Y$ in $\cT$. Since $\cT$ is a free $\cS$-module, it is a faithfully flat extension of $\cS$ and due to \cite[Theorem 13.3]{W79}, a faithfully flat extension of $\cR$. Since the image of $a \otimes b$ (and thus, the image of $x$) in $(\bG/\bT)(\cT)$ is $\iota(y^{-1}a \otimes yb)$, then the sheaf property of quotient maps is satisfied and therefore $\iota$ is a quotient map. Hence, $\iota$ is an isomorphism.

With the same arguments as in the proofs above, we see that the map
\begin{equation}
\delta \colon \bG/\bT \longrightarrow \bAut(\TtI_{m,n}),
\end{equation}
given by $\delta(a \otimes b)=\widetilde{L}_a\widetilde{R}_{b^\Tr}$,
defines a closed imbedding, and we claim that $\delta$ is a quotient map.
Take $\varphi\in \Aut_\cR\left((\VtI_{m,n})_\cR\right)$. Due to Proposition \ref{automorphismgroupscheme.I.rectangle}, to the isomorphism in Equation~\eqref{AutVIisomorphism}, and also to the sheaf property of quotient maps, it follows that there is a faithfully flat extension $f \colon \cR \to \cS$, $a\in \bGL_m(\cS)$ and $b\in \bGL_n(\cS)$ such that $\bAut(\VtI_{m,n})(f)(\varphi) = \widetilde{L}_a \widetilde{R}_b$. The following statements are equivalent:
$\varphi \in \Aut_\cS\big((\TtI_{m,n})_\cS\big)$ $\iff$ $\varphi^+ = \varphi^-$
$\iff$ $L_a R_b = L_{a^\Tr}^{-1} R_{b^\Tr}^{-1}$
$\iff$ $L_{a^\Tr a} R_{b b^\Tr} = \id$ $\iff$
$(a^\Tr a) x (b b^\Tr) = x$ for each $x \in \cM_{m,n}(\cS)$
$\iff$ $a^\Tr a = \lambda 1$, $b b^\Tr = \mu 1$ and $\lambda\mu = 1$, with $\lambda,\mu\in\cS^\times$
$\iff$ $a \in \GO_m(\cS)$ and $b \in \GO_n(\cS)$ with $m_a m_b = 1$. Since the sheaf property of quotient maps is satisfied, it follows that $\delta$ is a quotient map.
Since $\gamma = \delta \circ \iota$, it follows that $\gamma$ is well-defined and an isomorphism too.

Case 2) is proven similarly, since $\bmu_2^{(+)}(\cR) \leq \Aut_\cR\big((\TtI_n)_\cR\big)$, and case 3) is trivial.
\end{proof}

\begin{corollary}\label{c:fin1}
Let $1 < n \in\NN$. Then
\begin{align}
& \bAut(\ThI_n) \simeq \overline{\bAut}^{(+)}\big(\cM_n(\FF)\big) \times \bmu_2, \\
& \bAut\big( \cM_n(\FF)^{(+)} \big) = \overline{\bAut}^{(+)}\big(\cM_n(\FF)\big).
\end{align}
For the trivial case we have $\bAut(\ThI_1) \simeq \bmu_2$ and
$\bAut\big( \cM_1(\FF)^{(+)} \big) \simeq \mathbf{1}$ (the trivial group scheme).
\end{corollary}
\begin{proof}
Consequence of Prop.~\ref{schemesJandJTS} and Equation~\eqref{structureMnPlus}.
\end{proof}

\begin{corollary}\label{c:fin} For every $n>1$, $\TtI_n$ is not isomorphic to $\ThI_n$.
\end{corollary}
\begin{proof}
It suffices to prove that the respective automorphism group schemes are not isomorphic. In order to prove that, we will show that $\Lie((\bOrt_n\otimes_{\bmu_2}\bOrt_n)\rtimes \bmu_2)\simeq \mathfrak{so}_n(\FF)\oplus\mathfrak{so}_n(\FF)$ and that $\Lie\left(\overline{\bAut}^{(+)}\big(\cM_n(\FF)\big) \times \bmu_2\right)\simeq \mathfrak{\gl}_n(\FF)$. In order to do so, we should notice first that if $\bG$ and $\bH$ are smooth affine group schemes, then, \cite[22.13]{KMRT98} implies that $\Lie(\bG/\bH)\simeq \Lie(\bG)/\Lie(\bH)$. Also, as a consequence of \cite[22.12]{KMRT98}, if $\bG$ and $\bH$ are smooth, $\bG\rtimes \bH$ is smooth. Thus, again, from  \cite[22.13]{KMRT98}, we get an exact sequence of vector spaces
$$ 0\to \Lie(\bG)\to \Lie(\bG\rtimes \bH)\to \Lie(\bH)\to 0, $$
where the morphisms are algebra morphisms, implying that if $\Lie(\bG)=0$, then, $\Lie(\bG\rtimes \bH)\simeq \Lie(\bH)$ and if $\Lie(\bH)=0$, then, $\Lie(\bG\rtimes \bH)\simeq \Lie(\bG)$. Finally, from \cite[21.4]{KMRT98}, we get that for any affine group schemes $\bG$ and $\bH$, $\Lie(\bG\times \bH)\simeq \Lie(\bG)\oplus \Lie(\bH)$. With all this, the calculation of the Lie algebras follows from the fact that  $(\bOrt_n\otimes_{\bmu_2}\bOrt_n)\rtimes \bmu_2\simeq\left((\bOrt_n\times\bOrt_n)/\bmu_2\right)\rtimes\bmu_2$ and the fact that $\overline{\bAut}^{(+)}\big(\cM_n(\FF)\big) \times \bmu_2\simeq \left(\left(\bGL_n/\bG_m\right)\rtimes \bmu_2\right)\times \bmu_2$.
	
\end{proof}

\begin{remark}
	This last corollary would also follow from \cite[III.1]{S85} by taking the parameters $\Phi=\id$ and $A=\mathrm{I}$ to get  $\TtI_n$ and taking $\Phi$ as the transposition and $A=\mathrm{I}$ to get $\ThI_n$.
\end{remark}
\begin{remark}
Recall from \cite[Ex.~4.7]{A22} that there is a decomposition, as a tensor product of metric generalized Jordan pairs, given by $\VhI_{m,n} \cong \VhI_{1,m} \otimes \VhI_{1,n} \otimes \cV_{-2}$,
where the metrics of $\VhI_{m,n}$, $\VhI_{1,m}$ and $\VhI_{1,n}$ are their generic traces.
The term $\cV_\lambda$, with $\lambda \in \FF$, denotes a $1$-dimensional metric generalized Jordan pair with a metric $b$, which implies that $\bAut(\cV_\lambda, b) \simeq \bG_m$. In general, for a tensor product of metric generalized Jordan pairs we have that $\bAut(\cV, b_1) \otimes_{\bG_m} \bAut(\cW, b_2) \lesssim  \bAut(\cV \otimes \cW, b_1 \otimes b_2)$, see \cite[Prop.~4.3-4)]{A22}. Consequently, we have
\begin{equation} \label{tensor.decomp}
\bAut(\VhI_{1,m}, t) \otimes_{\bG_m} \bAut(\VhI_{1,n}, t) \lesssim  \bAut(\VhI_{m,n}, t).
\end{equation}
Note that Prop.~\ref{automorphismgroupscheme.I.rectangle} actually shows that there is an isomorphism in Equation~\eqref{tensor.decomp} if $m < n$. On the other hand, if $m = n$, Prop.~\ref{automorphismgroupscheme.I.square} shows that there is an additional factor $\bmu_2$ appearing in $\bAut(\VhI_n, t)$, which produces the transposition automorphisms; this corresponds to the swapping automorphisms $\bmu_2$ of $\VhI_{1,n} \otimes \VhI_{1,n}$, which interchange the two copies of $\VhI_{1,n}$.
\end{remark}

\noindent\textbf{Acknowledgements}.
The authors are thankful to Irene Paniello, for giving access to useful references, and to Alberto Elduque, for providing useful feedback. Thanks are also due to the anonymous referee, for the careful reading of our manuscript, and the many useful comments and corrections.


\end{document}